\documentclass[runningheads]{llncs}
\usepackage{graphicx}
\usepackage{breakcites}
\usepackage{amsfonts}
\usepackage{amsmath,amssymb}
\usepackage{algorithm}
\usepackage{algpseudocode}
\usepackage{pifont}
\usepackage{float}
\usepackage{graphicx}
\usepackage{caption}
\usepackage{subcaption}

\newcommand{\circledOne}{\text{\ding{172}}}
\newcommand{\circledTwo}{\text{\ding{173}}}
\newcommand{\circledThree}{\text{\ding{174}}}

\DeclareMathOperator*{\argmin}{arg\,min}

\usepackage[breaklinks=true,colorlinks=true,linkcolor=red, citecolor=blue, urlcolor=blue]{hyperref}%

\usepackage{xcolor}
\newcommand{\dotprod}[2]{\langle #1,#2 \rangle} % Скалярное произведение

\begin{document}
\title{Zero-Order Stochastic Conditional Gradient Sliding Method for Non-smooth Convex Optimization\thanks{The work was supported by the Ministry of Science and Higher Education of the Russian Federation (Goszadaniye) 075-00337-20-03, project No. 0714-2020-0005.}}
\titlerunning{Zero-Order Stochastic Conditional Gradient Sliding Method}
\author{Aleksandr Lobanov\inst{1,2,3}\orcidID{0000-0003-1620-9581} \and
Anton Anikin\inst{4}\orcidID{0000-0002-7681-2481} \and
Alexander Gasnikov\inst{1,5,6}\orcidID{0000-0002-7386-039X} \and
Alexander Gornov\inst{4}\orcidID{0000-0002-8340-5729}
\and
Sergey Chukanov\inst{7}}
\authorrunning{A. Lobanov et al.}
\institute{Moscow Institute of Physics and Technology, Dolgoprudny, Russia\\ 
\email{lobanov.av@mipt.ru},
\email{gasnikov.av@mipt.ru}\\
\and 
 ISP RAS Research Center for Trusted Artificial Intelligence, Moscow, Russia
 \and
 Moscow Aviation Institute, Moscow, Russia\\
\and 
Matrosov Institute for System Dynamics and Control Theory, Irkutsk, Russia\\
% \email{\{anikin,gornov\}@icc.ru}\\
\email{anikin@icc.ru},
\email{gornov@icc.ru}\\
\and Institute for Information Transmission Problems RAS, Moscow, Russia \and Caucasus Mathematical Center, Adyghe State University, Maikop, Russia \and Federal Research Center "Computer Science and Control" RAS, Moscow, Russia\\
\email{chukanov47@mail.ru}
}
\maketitle              % typeset the header of the contribution
\begin{abstract}
The conditional gradient idea proposed by Marguerite Frank and Philip Wolfe in 1956 was so well received by the community that new algorithms (also called Frank--Wolfe type algorithms) are still being actively created. In this paper, we study a non-smooth stochastic convex optimization problem with constraints. Using a smoothing technique and based on an accelerated batched first-order Stochastic Conditional Gradient Sliding method, we propose a novel gradient-free Frank--Wolfe type algorithm called Zero-Order Stochastic Conditional Gradient Sliding (ZO-SCGS). This algorithm is robust not only for the class of non-smooth problems, but surprisingly also for the class of smooth black box problems, outperforming the SOTA algorithms in the smooth~case~in~term oracle calls. In practical experiments we confirm our theoretical~results.

\end{abstract}
\section{Introduction}
    %Условный градиент
    The history of the conditional gradient method begins with the Frank--Wolfe algorithm proposed in 1956 \cite{Frank_Wolfe_1956}. Marguerite Frank and Philip Wolfe proposed an alternative to the gradient descent method for solving a class of quadratic constrained optimization problems that uses linear optimization on a convex compact set, avoiding projection. A little later in 1966, Evgenii Levitin and Boris Polyak in~\cite{Levitin_Polyak_1966} investigated the Frank--Wolfe method (named Conditional Gradient), obtaining the rate of convergence and showed that this rate is optimal for the class of smooth convex problems and for all algorithms that use linear minimization oracle.
    % the lower bound  that mathches 
    % s of the convergence rate, which are unimprovable.
    Since then, the conditional gradient algorithm has~gained much interest in the community, because in some cases it is computationally cheaper to solve the linear minimization problem over the feasible set (thereby guaranteeing a presence over the feasible set) than to perform a projection over the feasible set. Currently, the conditional gradient method is actively used in~solving practical problems of network routing~\cite{LeBlanc_1985,Mitradjieva_2013,Gasnikov_2020}, matrix completion \cite{Freund_2017,Garber_2022}, as well as in problems of machine learning~\cite{Jaggi_2011,Negiar_2020}, federated learning~\cite{Dadras_2022}, online optimization~\cite{Hazan_2016,Chen_2019,Garber_2020}, standard optimization~\cite{Goldfarb_2017,Mokhtari_2020,Garber_2021} and huge-scale optimization~\cite{bubeck2015,cox2017,anikin2022efficient}. 
    
    %Безградиентные алгоритмы
    However, as far as we know, there are no gradient-free algorithms (based on the conditional gradient method) to solve the black box problem in the non-smooth case. Where the black box problem means that only the zero-order oracle~\cite{Rosenbrock_1960} is available to us, i.e. we have access to the value of the objective function, not its gradient. This class of problems is a particular case of the practical problems above, when the gradient calculation procedure is too expensive~\cite{Saha_2011,Akhavan_2020} or not available at all \cite{Bubeck_2019,Dvinskikh_2022}. Already in November 2022, a survey appeared~\cite{Gasnikov_2022}, which provides various techniques for creating optimal gradient-free algorithms (based generally on accelerated batched first-order methods) to solve the black-box problem. The optimal for a gradient-free algorithm is usually understood by three criteria: iteration complexity, oracle complexity, and maximum level of~adversary noise. Thus, by choosing the accelerated batched conditional gradient method and using the smoothing technique from the survey, it is possible to develop a gradient-free algorithm to solve black-box problem~in~non-smooth~case.
    
    %О работе
    In this paper, we focus on black-box problems in the non-smooth case, namely, non-smooth convex stochastic optimization problems. To solve this problem, we use a smoothing scheme approach with $l_2$ randomization. Based on the accelerated batched conditional gradient method, also known as the Stochastic Conditional Gradient Sliding Method from \cite{Lan_2016}, we create an algorithm and derive optimal estimates: iteration complexity, oracle complexity, and maximum adversary noise level. As far as we know, this is the first gradient-free algorithm for solving a non-smooth convex optimization problem. We show in theory that Zero-Order Stochastic Conditional Gradient Sliding Method outperforms the oracle complexity of gradient-free algorithms (which are state of the art algorithms) in a smooth setting, which is a surprising fact. In practical experiments we confirm our theoretical results.
    
    \subsection{Our Contributions}
    Our contributions can be summarized as:
    \begin{itemize}
        \item[$\bullet$]  We present the first gradient-free algorithm based on the conditional gradient method "Zero-Order Stochastic Conditional Gradient Sliding Method" (ZO-SCGS) for solving a non-smooth convex stochastic optimization problem with constraints.
        
        \item[$\bullet$] Our theoretical results show that the algorithm is robust for black-box problems not only in the non-smooth case, but also for the smooth setting case. That is, our algorithm outperforms state of the art algorithms on oracle calls. In particular, the SOTA algorithm Zero-Order Conditional Gradient Method (ZSCG) from \cite{Balasubramanian_2022} has an estimation of oracle complexity $\sim \varepsilon^{-3}$, while our algorithm has an estimation of oracle complexity $\sim \varepsilon^{-2}$.
        
        \item[$\bullet$] We empirically test our theoretical results by comparing the Zero-Order Stochastic Conditional Gradient Sliding Method (ZO-SCGS) with the Zero-Order Conditional Gradient Method (ZSCG) on a model case in a smooth setting. We explain the reason for the advantage of the proposed algorithm.
    \end{itemize}
    \setcounter{footnote}{3}
    \subsection{Paper Organization}
    This paper has the following structure. In Section \ref{section:Related_Works} we provide related works. In Section \ref{section:Setup} we consider the formulation of the problem. We present the novel gradient-free algorithm in Section \ref{section:Main_Result}. In Section \ref{section:Discussion} we discuss the theoretical results obtained. We verify our results with a model experiment in Section \ref{section:Experiments}. While Section \ref{section:Conclusion} concludes the paper. We provide a detailed proof of the Theorem~\ref{theorem_1} in the supplementary materials (Appendix \ref{Appendix})\footnote{The full version of this article, which includes the Appendix \ref{Appendix} can be found at the~following link: \url{https://arxiv.org/abs/2303.02778}.}.
    
\section{Related Works}\label{section:Related_Works}
    %Методы условного градиента
    \textbf{Conditional gradient methods.}\;\;\; There are many works \cite{Goldfarb_2017,Mokhtari_2020,Lan_2016,Balasubramanian_2022,Hazan_2016_June,Yurtsever_2019,Zhang_2020,Duchi_2011,McMahan_2010,Combettes_2020} in the field of conditional gradient methods research. The latest research results in this area are presented in a recent survey on conditional gradient methods \cite{Braun_2022}. For instance, the Stochastic Frank--Wolf algorithm from \cite{Hazan_2016_June}, which is a generalization of the Frank--Wolf algorithm to stochastic optimization by replacing the gradient in the update with its stochastic approximation, requires $\sim \varepsilon^{-3}$ calls of stochastic gradients and performing $\sim \varepsilon^{-1}$ linear optimization. Also, for instance, the Stochastic Away Frank--Wolfe algorithm from~\cite{Goldfarb_2017}, which is derived from combining the Away-Step Stochastic Frank--Wolfe algorithm \cite{Guelat_1986} and the Pairwise Stochastic Frank--Wolfe algorithm \cite{Lacoste_2015}, requires~$\sim \varepsilon^{-4} \log^6 \left(\varepsilon^{-1}\right) $ calls of stochastic gradients and performing~$\sim \log \left(\varepsilon^{-1}\right)$ linear optimization. In another work \cite{Mokhtari_2020}, the Momentum Stochastic Frank--Wolf algorithm, which is obtained from the Stochastic Frank–-Wolfe algorithm by replacing the gradient estimator with the momentum estimator, requires~$\sim \varepsilon^{-3}$ calls of stochastic gradients and linear optimization. And in \cite{Lan_2016} the Stochastic Conditional Gradient Sliding algorithm was proposed, which is an accelerated batched method and requires $\sim \varepsilon^{-2}$ calls of stochastic gradients and performing $\sim \varepsilon^{-1}$ linear optimization. The above algorithms solve the problem of convex stochastic optimization and are first-order methods, but the Zeroth-Order Stochastic Conditional Gradient Method from \cite{Balasubramanian_2022}, which solves the black box problem in the smooth case, requires $\sim \varepsilon^{-3}$ calls of stochastic gradients and performing $\sim \varepsilon^{-1}$ linear optimization. In this paper, we choose the accelerated batched first order method: Stochastic Conditional Gradient Sliding algorithm from \cite{Lan_2016} as the basis for creating a novel gradient-free algorithm, since it has the best number of stochastic gradient calls presented. We will compare the efficiency of our algorithm to the Zeroth-Order Stochastic Conditional Gradient Method from \cite{Balasubramanian_2022}, which is one of the SOTA algorithms.\\
    %+может быть таблицу встатвить...
    %Техники сглаживания
    \textbf{Gradient-free methods.}\;\;\; The research field of gradient-free algorithms can be traced back to at least 1952 \cite{Kiefer_1952}. Recent works \cite{Dvinskikh_2022,Vaswani_2019,Dvurechensky_2021,Bach_2016,Gasnikov_ICML,Akhavan_2022,Lobanov_2022} are heavily focused on creating optimal gradient-free algorithms based on three criteria: iteration complexity, oracle complexity, and maximum level of adversary noise. For black-box problems, a gradient approximation is usually used instead of an exact gradient in first-order algorithms. For instance, work~\cite{Vaswani_2019} investigated gradient approximation via coordinate-wise randomization, and work~\cite{Dvurechensky_2021} investigated gradient approximation via random search randomization. Also, for instance, in \cite{Bach_2016} the gradient approximation via a "kernel-based" approximation is studied, the feature of which is to take into account the advantages for the case of increased smoothness. Some works use smoothing schemes via $l_1$ or $l_2$ randomization. For instance, paper \cite{Akhavan_2022} studied $l_1$ randomization as an alternative to the exact gradient for solving smooth optimization problems. Another paper~\cite{Gasnikov_ICML} explained the advantages of solving non-smooth problems using a smoothing scheme with $l_2$ randomization. And in \cite{Lobanov_2022} the smoothing scheme through $l_1$ randomization for non-smooth optimization problems is investigated and it is shown that in practice there are no significant advantages of $l_1$ randomization over $l_2$ randomization. In this paper, we use a smoothing scheme with $l_2$ randomization to create a gradient-free algorithm for solving a non-smooth convex stochastic optimization problem.

\section{Setup}\label{section:Setup}
    We study a non-smooth convex stochastic optimization problem with constraints
    \begin{equation}\label{init_problem}
        f^* := \min_{x\in Q} \left[ f(x):= \mathbb{E}_\xi \left[ f(x,\xi) \right] \right]
    \end{equation}
    where $Q \subseteq \mathbb{R}^d$ is a convex compact set and $f: Q \rightarrow \mathbb{R}$ is a convex function. This problem is also known as the black box problem, where a zero-order (gradient-free) oracle returns a function value $f(x,\xi)$ at the requested point $x$, possibly with some adversarial noise $\delta(x)$. We now formally introduce the definition of a gradient-free oracle.
    \begin{definition}[Gradient-free oracle]\label{def:gradient_free_oracle}
        Let gradient-free oracle returns a noise value of $f(x,\xi)$, i.e. for all $x \in Q$
        \begin{equation*}
            f_\delta(x,\xi) := f(x,\xi) + \delta(x). 
        \end{equation*}
    \end{definition}
    Next, we consider the assumptions we use in our theoretical results.
    \subsection{Assumptions}
    We assume that the function is Lipschitz continuous and is convex on set~$Q_\gamma$.
    \begin{assumption}[Lipschitz continuity of the function]\label{ass:lipschitz_continuity}
        Function $f(x, \xi)$ is Lipschitz continuous with constant $M$, i.e. for all $x,y \in Q$:
        \begin{equation*}
            |f(y,\xi) - f(x, \xi)| \leq M(\xi) \| y - x \|_p.
        \end{equation*}
        Moreover, there exists a positive constant $M$ such that $\mathbb{E} \left[ M^2(\xi) \right] \leq M^2$. 
    \end{assumption}
    \begin{assumption}[Convexity on the set $Q_\gamma$]\label{ass:convexity_Q_gamma} Let $\gamma > 0$ a small number to be defined later and $Q_\gamma := Q + B_2^d(\gamma)$, then the function $f$ is convex on the set $Q_\gamma$.
    \end{assumption}
    
    We also assume that adversarial noise is bounded.
    \begin{assumption}[Boundedness of noise]\label{ass:bound_noise} For all $x \in Q$, it holds $|\delta(x)| \leq \Delta$.
    \end{assumption}
    
    Our Assumption \ref{ass:lipschitz_continuity} of a Lipschitz continuity of the function is similar as in~\cite{Gasnikov_ICML} and generalizes to a stochastic setting. For the special case when $p=2$ we use the notation $M_2$ for the Lipschitz constant (see e.g. \cite{Dvinskikh_2022}). Assumption \ref{ass:convexity_Q_gamma} is quite standard in the literature (see e.g. \cite{Yousefian_2012,Risteski_2016}). We used $l_2$-ball here since we use $l_2$~randomization in this paper. In more general the formulation of the assumption depends on the choice of gradient approximation (see e.g. \cite{Lobanov_2022}). So much prior work in the context of stochastic optimization often assumed the boundedness of stochastic or deterministic noise (such as e.g. \cite{Akhavan_2020,Stich_2020,Vasin_2021}). In Assumption \ref{ass:bound_noise}, we consider bounded deterministic noise.
    
    \subsection{Notation}
    We use $\dotprod{x}{y}:= \sum_{i=1}^{d} x_i y_i$ to denote standard inner product of $x,y \in \mathbb{R}^d$, where $x_i$ and $y_i$ are the $i$-th component of $x$ and $y$ respectively. We denote $l_p$-norms (for~$p \geq 1$) in $\mathbb{R}^d$ as $\| x\|_p := \left( \sum_{i=1}^d |x_i|^p \right)^{1/p}$. Particularly for $l_2$-norm in $\mathbb{R}^d$ it follows $\| x \|_2 := \sqrt{\dotprod{x}{x}}$. We denote $l_p$-ball as $B_p^d(r):=\left\{ x \in \mathbb{R}^d : \| x \|_p \leq r \right\}$ and $l_p$-sphere as $S_p^d(r):=\left\{ x \in \mathbb{R}^d : \| x \|_p = r \right\}$. Operator $\mathbb{E}[\cdot]$ denotes full mathematical expectation. We notation $\tilde{O} (\cdot)$ to hide logarithmic factors. To define the diameter of the set $Q$ we introduce $D := \max_{x,y \in Q} \| x - y\|_p$.
    % Добавить текст...
    
\section{Main Result}\label{section:Main_Result}
    In this section, we present a novel algorithm (see Algorithm \ref{alg:ZOSCGS}) that is optimal in terms of iterative complexity, the number of gradient-free oracle calls, and the maximum value of adversarial noise. This algorithm is based on an accelerated first-order Stochastic Conditional Gradient Sliding (SCGD) method from \cite{Lan_2016}. This section is structured as follows: in Subsection \ref{Subsection:smoothing_intuition} we introduce the basic idea of the smoothing scheme, in Subsections \ref{Subsection:smoothing_approximation} and \ref{Subsection:l2_randomization} we consider the main elements of the smoothing scheme via $l_2$ randomization, and in Subsection \ref{Subsection:Zero_Order_Method} we present the new gradient-free method (see Algorithm \ref{alg:ZOSCGS} for more details).  \\
    We start with the main idea of solving problem \eqref{init_problem} via the smoothing scheme.

    \subsection{Smoothing scheme intuition}\label{Subsection:smoothing_intuition}
    The main idea of solving problem \eqref{init_problem} via the smoothing scheme is to replace the problem. That is, instead of solving the non-smooth problem we will solve its smoothed problem: 
    \begin{equation}\label{problem_smoothed}
        \min_{x \in Q} f_\gamma (x),
    \end{equation}
    where $f_\gamma$ a smooth approximation of the non-smooth function $f$, which we define below. Thus, to solve the smooth problem \eqref{problem_smoothed} it is sufficient to choose the accelerated batched algorithm $\textbf{A}(L_{f_\gamma},\sigma^2)$. Next, we introduce the assumptions of smoothness of the function $f_\gamma$ and bounded variance of the gradient~$\nabla f_\gamma(x,\psi)$.
    
    \begin{assumption}[$L_{f_\gamma}$-smoothness] \label{ass:L_smoothness}
        Function $f_\gamma(x)$ is differentiable and there exists a constant $L_{f_\gamma} \geq 0$ such that for $x,y \in Q$:
        \begin{equation*}
            \| \nabla f_\gamma (y) - \nabla f_\gamma(x) \|_q \leq L_{f_\gamma} \| y - x \|_p.
        \end{equation*}
    \end{assumption}
    
    \begin{assumption}[Bounded variance and unbiased]\label{ass:Bounded_variance}
        Gradient~$\nabla f_\gamma(x, \xi)$ has bounded variance such that for $x \in Q$:
        \begin{equation*}
            \mathbb{E}_{\psi}\left[ \| \nabla_x f_\gamma(x, \psi) - \nabla f_\gamma(x) \|^2_q \right] \leq \sigma^2, \;\;\; \mathbb{E}_{\psi}\left[ \nabla f_\gamma(x,\psi) \right] = \nabla f_\gamma(x).
        \end{equation*}
    \end{assumption}

    Assumptions \ref{ass:L_smoothness} and \ref{ass:Bounded_variance} are quite common in the literature (see e.g. \cite{Lan_2016,Stich_2020,Gorbunov_2020}). Here $q$ is such that $1/p + 1/q = 1$. And a random variable~$\psi$~we define below. 
    
    The connection between Problems \eqref{init_problem} and \eqref{problem_smoothed} is as follows: to solve a non-smooth problem with $\varepsilon$-accuracy, it is necessary to solve a smooth problem with $(\varepsilon/2)$-accuracy, where $\varepsilon$-suboptimality is the accuracy of the solution in terms of expectation (see Appendix \ref{Appendix} for the proof of this statement). So, to solve Problem \eqref{init_problem} (under the Assumption \ref{ass:L_smoothness} and \ref{ass:Bounded_variance}) with Algorithm $\textbf{A}(L_{f_\gamma},\sigma^2)$, we need to know the gradient of the smoothed function $\nabla f_\gamma(x,\psi)$, $L_{f_\gamma}$-smoothness constant, and the variance estimate $\sigma^2$.
    
    In the following subsections we will define these elements.
    
    \subsection{Smooth approximation}\label{Subsection:smoothing_approximation}
    Since problem \eqref{init_problem} is non-smooth, we introduce a smooth approximation of the non-smooth function $f$ as follows:
    \begin{equation}\label{f_gamma}
        f_\gamma(x) := \mathbb{E}_{\tilde{e}} \left[ f(x + \gamma \tilde{e}) \right],
    \end{equation}
    where $\gamma > 0$ is smoothing parameter, $\tilde{e}$ is random vector uniformly distributed on $B_2^d(\gamma)$. Here $f_\gamma(x) := \mathbb{E} \left[ f(x, \xi) \right]$. The following lemma provides the connection between non-smooth function $f$ and smoothed function $f_\gamma$.
    
    \begin{lemma}\label{Lemma:connect_f_with_f_gamma}
        Let Assumptions \ref{ass:lipschitz_continuity}, \ref{ass:convexity_Q_gamma} it holds, then for all $x \in Q$ we have
        \begin{equation*}
            f(x)\leq f_\gamma(x) \leq f(x) + \gamma M_2.
        \end{equation*}
    \end{lemma}
    \begin{proof}
    For the first inequality we use the convexity of the function $f(x)$
    \begin{equation*}
        f_\gamma(x) = \mathbb{E}_{\Tilde{e}} \left[ f(x + \gamma \Tilde{e}) \right] \geq \mathbb{E}_{\Tilde{e}} \left[ f(x) + \dotprod{\nabla f(x)}{\gamma \Tilde{e}}) \right] = \mathbb{E}_{\Tilde{e}} \left[ f(x) \right] = f(x).
    \end{equation*}
    For the second inequality we have
    \begin{eqnarray*}
        | f_\gamma (x)- f(x) | = | \mathbb{E}_{\Tilde{e}} \left[ f(x + \gamma \Tilde{e}) \right] - f(x) | &\leq& \mathbb{E}_{\Tilde{e}} \left[ | f(x + \gamma \Tilde{e}) - f(x) | \right]\\
        &\leq& \gamma M_2 \mathbb{E}_{\Tilde{e}} \left[ \| \Tilde{e} \|_2 \right] \leq \gamma M_2,
    \end{eqnarray*}
    using the fact that $f$ is $M_2$-Lipschitz function.
    \end{proof}\qed
    
    The next lemmas provide properties of the smoothed function $f_\gamma$.
    \begin{lemma}\label{Lemma:M_lipschitz_continuity}
        Let Assumptions \ref{ass:lipschitz_continuity}, \ref{ass:convexity_Q_gamma} it holds, then for $f_\gamma(x)$ from \eqref{f_gamma} we have
        \begin{equation*}
            |f_\gamma(y) - f_\gamma(x) | \leq M \| y - x \|_p, \;\;\; \forall x,y \in Q.
        \end{equation*}
    \end{lemma}
    \begin{proof} Using $M$-Lipschitz continuity of function $f$ we obtain
        \begin{eqnarray*}
            | f_\gamma (y)- f_\gamma(x) | \leq \mathbb{E}_{\Tilde{e}} \left[ | f(y + \gamma \Tilde{e})  - f(x + \gamma \Tilde{e}) | \right] \leq M \| y - x \|_p.
        \end{eqnarray*}
    \end{proof}\qed
    
    %\newpage%%%%%%%%%%%%%%%%%%%%%%%%%%%%%%%%%%%%%%%%%
    
    \begin{lemma}[Theorem 1, \cite{Gasnikov_ICML}]\label{Lemma:Lipschitz_gradient}
        Let Assumptions \ref{ass:lipschitz_continuity}, \ref{ass:convexity_Q_gamma} it holds, then $f_\gamma(x)$ has $L_{f_\gamma} = \frac{\sqrt{d}M}{\gamma}$-Lipschitz gradient
        \begin{equation*}
            \| \nabla f_\gamma(y) - \nabla f_\gamma(x) \|_q \leq L_{f_{\gamma}} \| y - x \|_p, \;\;\; \forall x,y \in Q.
        \end{equation*}
    \end{lemma}
    
    \subsection{Gradient via $l_2$ randomization}\label{Subsection:l2_randomization}
    The gradient of $f_\gamma(x,\xi)$ can be estimated by the following approximation:
    \begin{equation}
        \nabla f_\gamma(x, \xi, e) = \frac{d}{2 \gamma} \left( f_\delta(x+ \gamma e, \xi) - f_\delta(x - \gamma e, \xi) \right) e,
    \end{equation}
    where $f_\delta(x,\xi)$ is gradient-free oracle from Definition \ref{def:gradient_free_oracle}, $e$ is a random vector uniformly distributed on $S_2^d(\gamma)$. The following lemma provides properties of the gradient $\nabla f_\gamma(x,\xi,e)$.
    
    \begin{lemma}[Lemma 2, \cite{Lobanov_2022}]\label{Lemma:sigma}
         Gradient $\nabla f_\gamma(x,\xi,e)$ has bounded variance (second moment) for all $x \in Q$
        \begin{equation*}
           \mathbb{E}_{\xi, e} \left[ \| \nabla f_\gamma (x, \xi, e) \|^2_q \right] \leq \kappa(p,d) \left( d M_2^2 + \frac{d^2 \Delta^2}{\sqrt{2} \gamma^2}  \right), 
        \end{equation*}
        where $1/p + 1/q = 1$ and 
        \begin{equation*}
            \kappa(p,d) = \sqrt{2} \min \left\{ q, \ln d \right\} d^{1 - \frac{2}{p}}.
        \end{equation*}
    \end{lemma}
    \begin{remark}\label{Remark:sigma}
        Using the fact that the second moment is the upper estimate of the variance for the unbiased gradient and assuming that $\Delta$ is sufficiently small we obtain the following estimate of the variance from Lemma \ref{Lemma:sigma}:
        \begin{equation*}
            \sigma^2 \leq 2 \sqrt{2} \min \left\{ q, \ln d \right\} d^{2 - \frac{2}{p}}  M_2^2.
        \end{equation*}
    \end{remark}
    
    \subsection{Zero-Order Stochastic Conditional Gradient Sliding Method}\label{Subsection:Zero_Order_Method}
    Now we present gradient-free algorithm (see Algorithm \ref{alg:ZOSCGS}) to solve problem \eqref{init_problem}. We chose Stochastic Conditional Gradient Sliding Method as accelerated batched Algorithm $\textbf{A}(L_{f_\gamma},\sigma^2)$. Substituting the approximation of the gradient via $l_2$ randomization $\nabla f_\gamma(x, \xi, e)$ ($\nabla f_\gamma(x, \psi)$ from Subsection \ref{Subsection:smoothing_intuition}, where $\psi = (\xi, e)$  is not only the random value $\xi$, but also the randomization on the $l_2$-sphere $e$, which was introduced in Subsection \ref{Subsection:l2_randomization}) instead of the exact gradient, we obtain a new ZO-SCGD Algorithm \ref{alg:ZOSCGS} to solve the non-smooth problem \eqref{init_problem}.
    
    \begin{algorithm}
        \caption{Zero-Order Stochastic Conditional Gradient Sliding (ZO-SCGS)}\label{alg:ZOSCGS}
        \textbf{Input}: Start point $x_0 \in Q$, maximum number of iterations $N \in \mathbb{Z}_+$.\\
        \hspace*{\algorithmicindent} Let stepsize $ \zeta_k \in [0,1]$, learning rate $\eta_k > 0$, accuracies $\beta_k$, batch size $B_k \in \mathbb{Z}_+$,\\
        \hspace*{\algorithmicindent} smoothing parameter~$\gamma>0$. \\
        \textbf{Initialization}: Generate independently vectors $e_1, e_2, ...$ uniformly distributed on unit \\   
        \hspace*{\algorithmicindent} $l_2$-sphere, and set $y_0 \gets x_0$
        \begin{algorithmic}[1]
        \For{$k=1,...,N$}
        \State $z_k \gets (1-\zeta_k)x_{k-1} + \zeta_k y_{k-1}$
        \State Sample $\{ e_1,...,e_{B_k} \}$ and $ \{ \xi_1,...,\xi_{B_k} \}$ independently
        \State $g_k \gets \frac{1}{B_k} \sum_{i=1}^{B_k} \left[ \frac{d}{2 \gamma} \left( f(z_k + \gamma e_i, \xi_i) - f(z_k - \gamma e_i, \xi_i) \right) e_i \right]$
        \State $y_{k} \gets \text{CG}(g_k, y_{k-1}, \eta_k, \beta_k)$ \Comment{See CG in Algorithm \ref{alg:CG}}
        \State $x_{k} \gets (1 - \zeta_k)x_{k-1} + \zeta_k y_{k}$
        \EndFor
        \end{algorithmic}
        \textbf{Output}: $x_N$.
    \end{algorithm}
    Algorithm \ref{alg:ZOSCGS} has such parameters as number of iterations $N$, batch size~$B$, stepsize $ \zeta$, learning rate $\eta$, accuracies $\beta$. The recommendations for selecting these parameters can be found in Theorem \ref{theorem_1}. To prove theorem we also need to know the values of the following parameters: constant of Lipschitz gradient $L_{f_\gamma} = \frac{2 \sqrt{d} M M_2}{\varepsilon}$ (by substituting $\gamma = \varepsilon/(2 M_2)$ in Lemma \ref{Lemma:Lipschitz_gradient}), where constant of Lipschitz continuity $M$ defined in Lemma \ref{Lemma:M_lipschitz_continuity} under Assumption \ref{ass:lipschitz_continuity}, and estimate of the variance $\sigma^2~\leq~2~\sqrt{2} \min \left\{ q, \ln d \right\} d^{2 - \frac{2}{p}}  M_2^2$ (from Remark \ref{Remark:sigma}).
    
    Next theorem provides estimates of the convergence rate of Algorithm \ref{alg:ZOSCGS}.
    
    \begin{theorem}\label{theorem_1}
        Let $\varepsilon$ be desired accuracy to solve problem \eqref{init_problem} and $\gamma$ be chosen as $\gamma = \varepsilon/(2 M_2)$. Let function $f(x, \xi)$ satisfy the Assumptions \ref{ass:lipschitz_continuity}, \ref{ass:convexity_Q_gamma} and~\ref{ass:bound_noise}. Then Zero-Order Stochastic Conditional Gradient Sliding algorithm (see Algorithm \ref{alg:ZOSCGS}) with $\zeta_k~=~3/(k+3)$, $\eta_k = 8 \sqrt{d} M M_2/(\varepsilon(k+3))$, $\beta_k=2 \sqrt{d} M M_2 D^2/(\varepsilon(k+1)(k+2))$, and $B_k=\left\lceil \min\{q,\ln~d\}d^{1-\frac{2}{p}} (k+3)^3 \varepsilon^2 /(MD)^2 \right\rceil$ achieves $\mathbb{E}\left[ f(x_k) \right]-f^*\leq\varepsilon$ after 
        \begin{equation*}
            N = \mathcal{O} \left( \frac{d^{1/4}\sqrt{MM_2} D}{\varepsilon} \right), \;\;\; T = \mathcal{O} \left( \frac{\min\{q,\ln~d\}d^{2-\frac{2}{p}} M_2^2  D^2 }{ \varepsilon^2} \right)
        \end{equation*}
        number of iterations and gradient-free oracle calls respectively. 
    \end{theorem}
    See Appendix \ref{Appendix} for detailed proof.  
    
    The results of Theorem \ref{theorem_1} show that Zero-Order Stochastic Conditional Gradient Sliding algorithm converges with $\varepsilon$-accuracy in $ N \sim d^{1/4} \varepsilon^{-1}$~iterations. The number of solutions to linear optimization problems, also known as the linear minimization oracle (LMO), is $\mathcal{O}\left( \sqrt{d} \varepsilon^{-2} \right)$. Batch size $B_k \in \mathbb{Z}_+$ must be chosen integer, so in Theorem \ref{theorem_1} $\left\lceil \cdot \right\rceil$ denotes the whole part of the next integer number. The number of oracle calls $T$ requiring the Algorithm \ref{alg:ZOSCGS} to solve a non-smooth problem \eqref{init_problem} with $\varepsilon$-accuracy is $ T \sim  \min\{q,\ln~d\}d^{2-\frac{2}{p}} \varepsilon^{-2}$.
    \begin{remark}[Smooth setting] \label{Remark:smooth_setup}
        In Theorem~\ref{theorem_1}, we presented the convergence results of Algorithm~\ref{alg:ZOSCGS} in the non-smooth setting, since in this paper we focus on solving non-smooth convex stochastic optimization problems. However, the algorithm proposed in this paper is robust to the smooth setting as well. To obtain similar estimates of the algorithm for smooth setting, it is sufficient not to change constant of Lipschitz gradient (i.e., it is not necessary to substitute the value obtained in Lemma \ref{Lemma:Lipschitz_gradient}). Therefore, Algorithm \ref{alg:ZOSCGS} with parameters $\zeta_k~=~3/(k+3)$, $\eta_k = 4 L/(k+3)$, $\beta_k= L D^2/((k+1)(k+2))$, and $B_k=\left\lceil \min\{q,\ln~d\}d^{2-\frac{2}{p}} M_2^2 (k+3)^3 /(LD)^2 \right\rceil$ achieves $\mathbb{E}\left[ f(x_k) \right]-f^*\leq\varepsilon$ after $N \sim \varepsilon^{-1/2}$ iterations, performs $\sim \varepsilon^{-1}$ linear optimization and requires $T \sim \min\{q,\ln~d\}d^{2-\frac{2}{p}} \varepsilon^{-2}$ gradient-free oracle calls.
    \end{remark}
    \begin{remark}
        In Subsection \ref{Subsection:Zero_Order_Method}, we focus on obtaining optimal estimates of iterative~$N$ and oracle~$T$ complexities, so in proving the Theorem~\ref{theorem_1} we considered the case~$\Delta=0$. However, an optimal estimate of the maximum adversarial noise can be obtained by performing a similar convergence analysis of the Stochastic Conditional Gradient Sliding Method for the biased stochastic oracle (see example analysis in \cite{Gorbunov_2019}). For brevity, we omit this analysis, stating that the estimate of maximum adversarial noise is $\Delta \lesssim \varepsilon^2 d^{-1/2}$ for gradient-free algorithms created by applying smoothing scheme via $l_2$ randomization (see e.g. \cite{Dvinskikh_2022,Gasnikov_ICML}).
    \end{remark}

\section{Discussion}\label{section:Discussion}
    As far as we know, Zero-Order Stochastic Conditional Gradient Sliding (ZO-SCGS) is the first gradient-free conditional gradient-type algorithm that solves a non-smooth convex stochastic optimization problem \eqref{init_problem}. This algorithm, as Theorem \ref{theorem_1} shows, is robust for solving non-smooth black-box problems. But most interestingly, this algorithm is also robust for smooth black box problems, because it is superior in terms of the number of oracle calls to the state of the art algorithms. For instance, the Zeroth-Order Stochastic Conditional Gradient Method (ZSCG) from \cite{Balasubramanian_2022}, which is a SOTA algorithm, has the following oracle complexity of $T \sim d\; \varepsilon^{-3}$ in any setting, while Algorithm \ref{alg:ZOSCGS} has oracle complexity of $T \sim d\; \varepsilon^{-2}$ in the Euclidean setting $p=2$ $(q=2)$ and $T \sim \ln (d)\; \varepsilon^{-2}$ in the simplex setting $p=1$ $(q = \infty)$. One reason for the advantage of our algorithm may be that the ZSCG method uses Direct Finite Difference (FFD), while the ZO-SCGS method (see Algorithm \ref{alg:ZOSCGS}) uses Central Finite Difference (CFD). It is worth noting that \cite{Scheinberg_2022} explains why it is worth estimating the gradient via central finite difference. Another possible reason may be the choice of Gaussian smoothing instead of smoothing via $l_2$ randomization, because in practical examples it often happens that the algorithm whose gradient is approximated over $l_2$~randomization works better than the algorithm whose gradient is approximated over Gaussian smoothing. Last but not least, a possible reason is that the Zeroth-Order Stochastic Conditional Gradient Method (ZSCG) used the unaccelerated first-order Stochastic Frank--Wolfe (SFW) method of \cite{Hazan_2016} as its base. Since the Stochastic Conditional Gradient Method already has an estimate on the number of calls to the stochastic gradient as $\sim \varepsilon^{-3}$. It is hard to expect an improvement in estimate of oracle complexity when creating a gradient-free method based on it. Therefore, in this paper we created an optimal gradient-free method based on an accelerated batched first-order algorithm. So far we have observed theoretical advantages of Algorithm 1 (robust for solving non-smooth black box problems) in terms of oracle complexity over SOTA algorithms, which are robust for solving smooth black box problems. Therefore, in Section \ref{section:Experiments} we will verify our theoretical results with a model example of a convex stochastic optimization problem in a smooth setting.

\setcounter{footnote}{0}

\section{Experiments}\label{section:Experiments}
    In this section we focus on verifying our theoretical results obtained in Section~\ref{section:Main_Result} via experiments\footnote{Code repository link: \url{https://github.com/htower/zo-scgs}}. In particular, we  numerically compare the Zero-Order Stochastic Conditional Gradient Sliding Method (ZO-SCGS) proposed in this paper (see Algorithm \ref{alg:ZOSCGS}) with the Zeroth-Order Stochastic Conditional Gradient Method (ZSCG) from~\cite{Balasubramanian_2022}. We consider a standard model example of a black box problem in a smooth setting, which has the following form:
    \begin{equation*}
        \min_{x \in Q} f(x) := \frac{1}{2}\dotprod{x}{Ax} - \dotprod{b}{x},
    \end{equation*}
    where $Q = \left\{ x \in \mathbb{R}^d \;:\; \| x \|_1 = 1, x \geq 0  \right\}$ is a simplex set, $A \in \mathbb{R}^{d \times d}$ is a random positively determined matrix, $b \in \mathbb{R}^d$ is a vector such that $b = A x_*$, and $x_*$ is a solution to the problem $x_* = \argmin_{x \in Q} f(x)$. In all tests, the dimensionality of the problem is $d = 100$, we fix the maximum number of calls to the gradient-free oracle $T_{max} = 10^7$, and the parameters of the algorithms are taken according to theoretical recommendations: for instance, parameters for Algorithm~\ref{alg:ZOSCGS}, see Remark \ref{Remark:smooth_setup}, and parameters for Zeroth-Order Stochastic Conditional Gradient Method, see~\cite{Balasubramanian_2022}. In Figure~\ref{Figure_theory} we compare the ZO-SCGS method with the ZSCG method. In particular, Figure~\ref{Figure_theory_a} shows the dependence of the optimal error~($f(x_k) - f^*$) on the number of calls of the gradient-free oracle~$T$. And Figure~\ref{Figure_theory_b} examines the dependence of the optimal error~($f(x_k) - f^*$) on the number of iterations $N$. We observe that Algorithm~\ref{alg:ZOSCGS} significantly outperforms the ZSCG method in the number of oracle calls. Also, when the maximal value of the gradient-free oracle call is fixed, we see that the Algorithm~\ref{alg:ZOSCGS} is first inferior to the ZSCG method in the number of iterations.
    \begin{figure}[H]
    	\centering
    	\begin{subfigure}{0.49\textwidth}
    		\includegraphics[width=\textwidth]{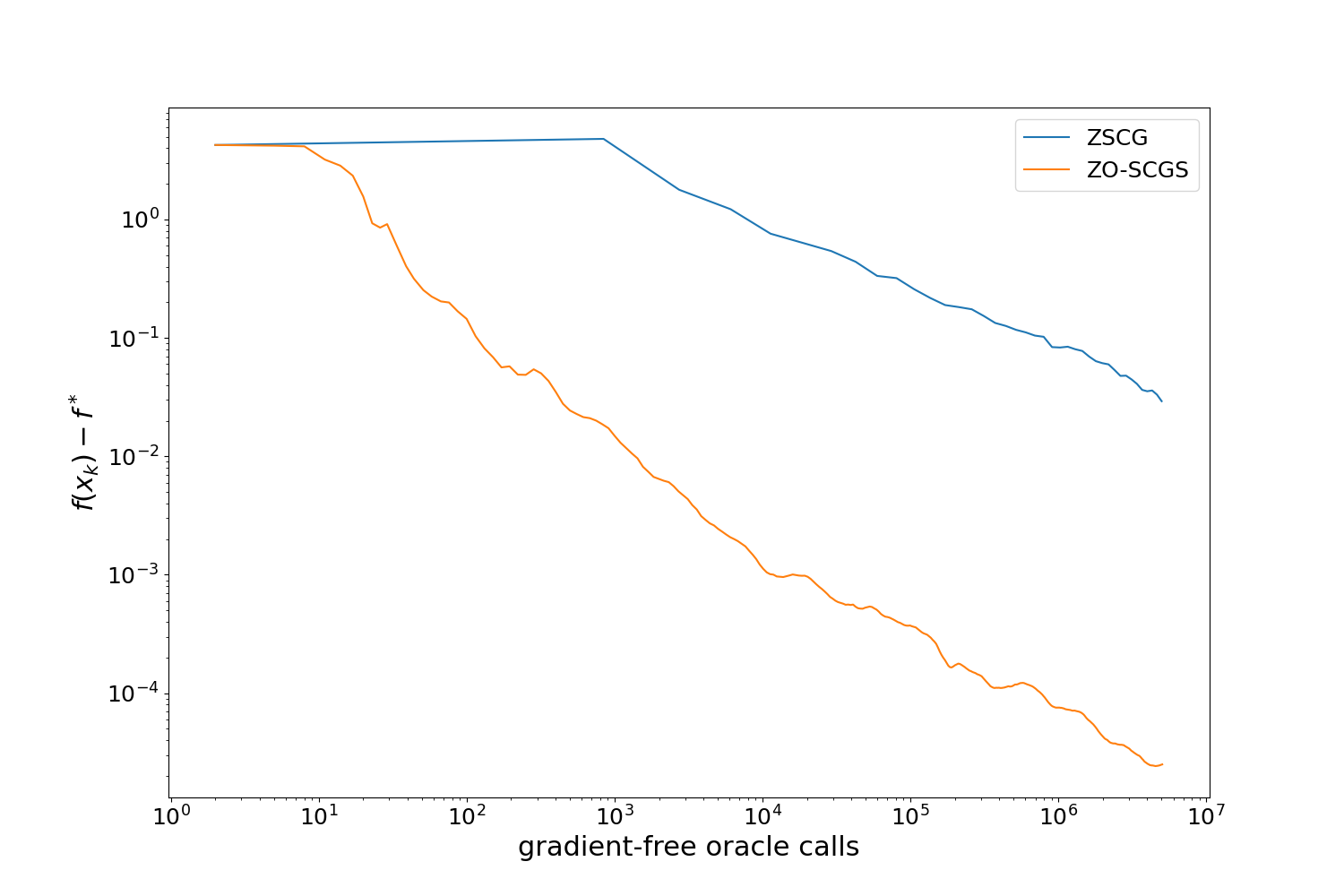}
    		\caption{}
    		\label{Figure_theory_a}
    	\end{subfigure}
    %%%%%%%%%%%%%%
    	\begin{subfigure}{0.49\textwidth}
    		\includegraphics[width=\textwidth]{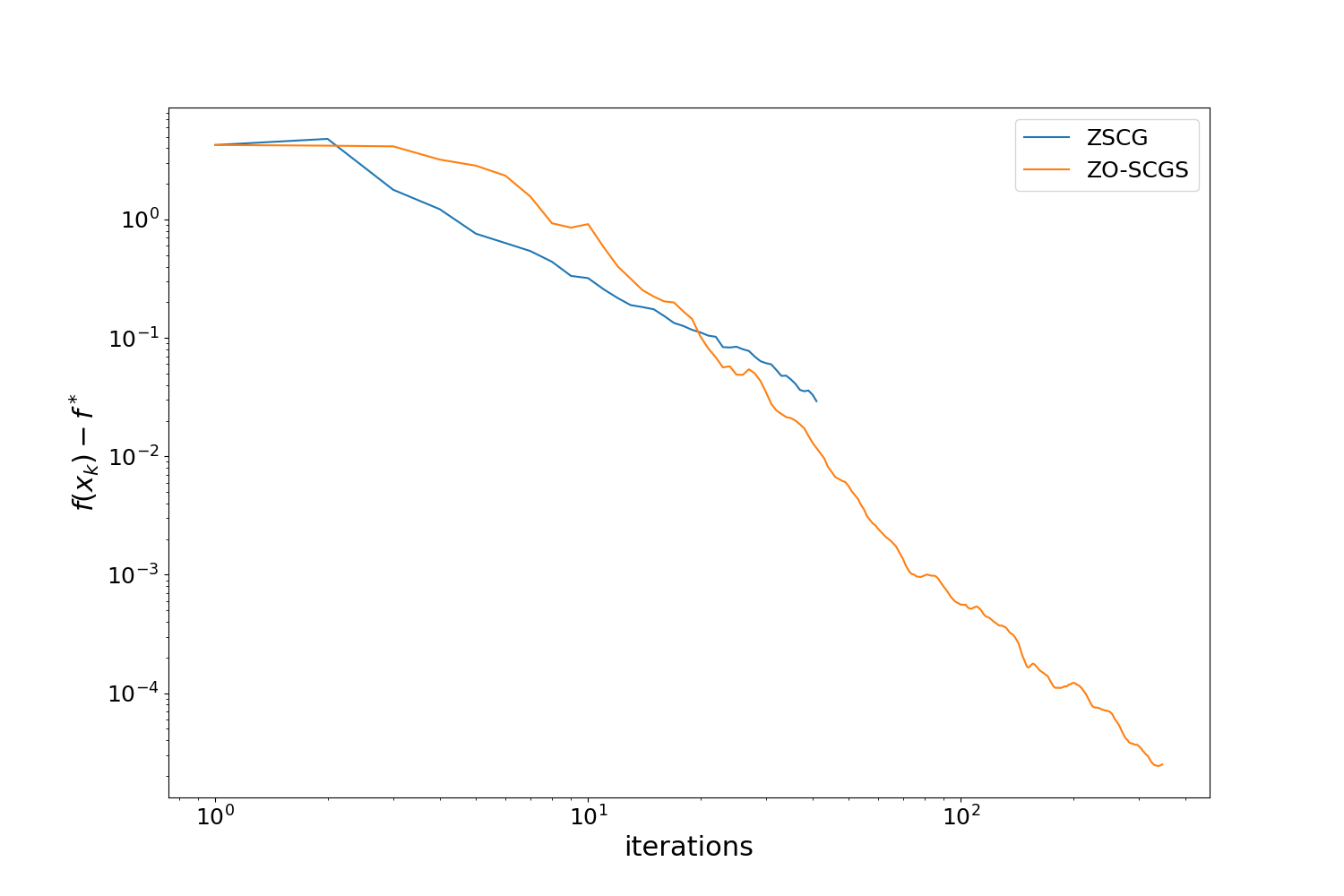}
    		\caption{}
    		\label{Figure_theory_b}
    	\end{subfigure}
    	\caption{Comparison of convergence result of Algorithm \ref{alg:ZOSCGS} with ZSCG method~\cite{Balasubramanian_2022}. }
    	\label{Figure_theory}
    \end{figure}

    According to theoretical estimates for the ZO-SCGS and ZSCG methods, the batch size should be taken at a large size, which is a disadvantage of these algorithms. In Figure \ref{Figure_fixed} we compare Algorithm 1 with ZSCG methods using the fixed batch-size~$B_k=100$. Figure~\ref{Figure_fixed_a} shows the dependence of the optimal error~($f(x_k) - f^*$) on the number of calls of the gradient-free oracle~$T$. And Figure~\ref{Figure_fixed_b} examines the dependence of the optimal error~($f(x_k) - f^*$) on the number of iterations $N$. We see that for a fixed (small) batch size, both algorithms have convergence, which is a positive result for practical experiments to use. We also see that ZO-SCGS and ZSCG methods require the same number of calls to the gradient-free oracle, since we have fixed the batch size in contrast to Figure \ref{Figure_theory}. We can also observe that Algorithm 1 significantly outperforms the method both in the number of to gradient-free oracle calls  and in iterations.
    \begin{figure}[H]
    	\centering
    	\begin{subfigure}{0.49\textwidth}
    		\includegraphics[width=\textwidth]{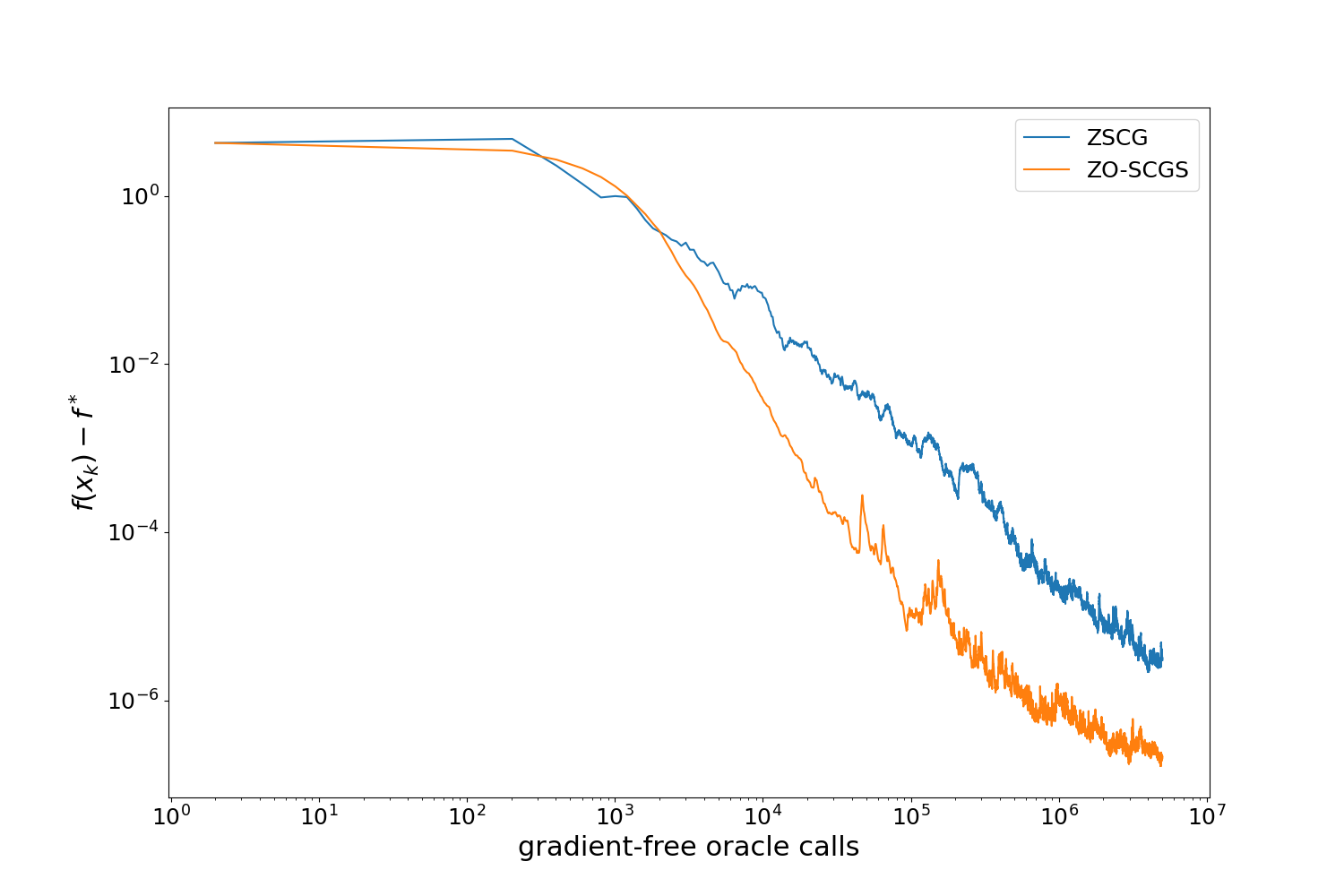}
    		\caption{}
    		\label{Figure_fixed_a}
    	\end{subfigure}
    %%%%%%%%%%%%%%
    	\begin{subfigure}{0.49\textwidth}
    		\includegraphics[width=\textwidth]{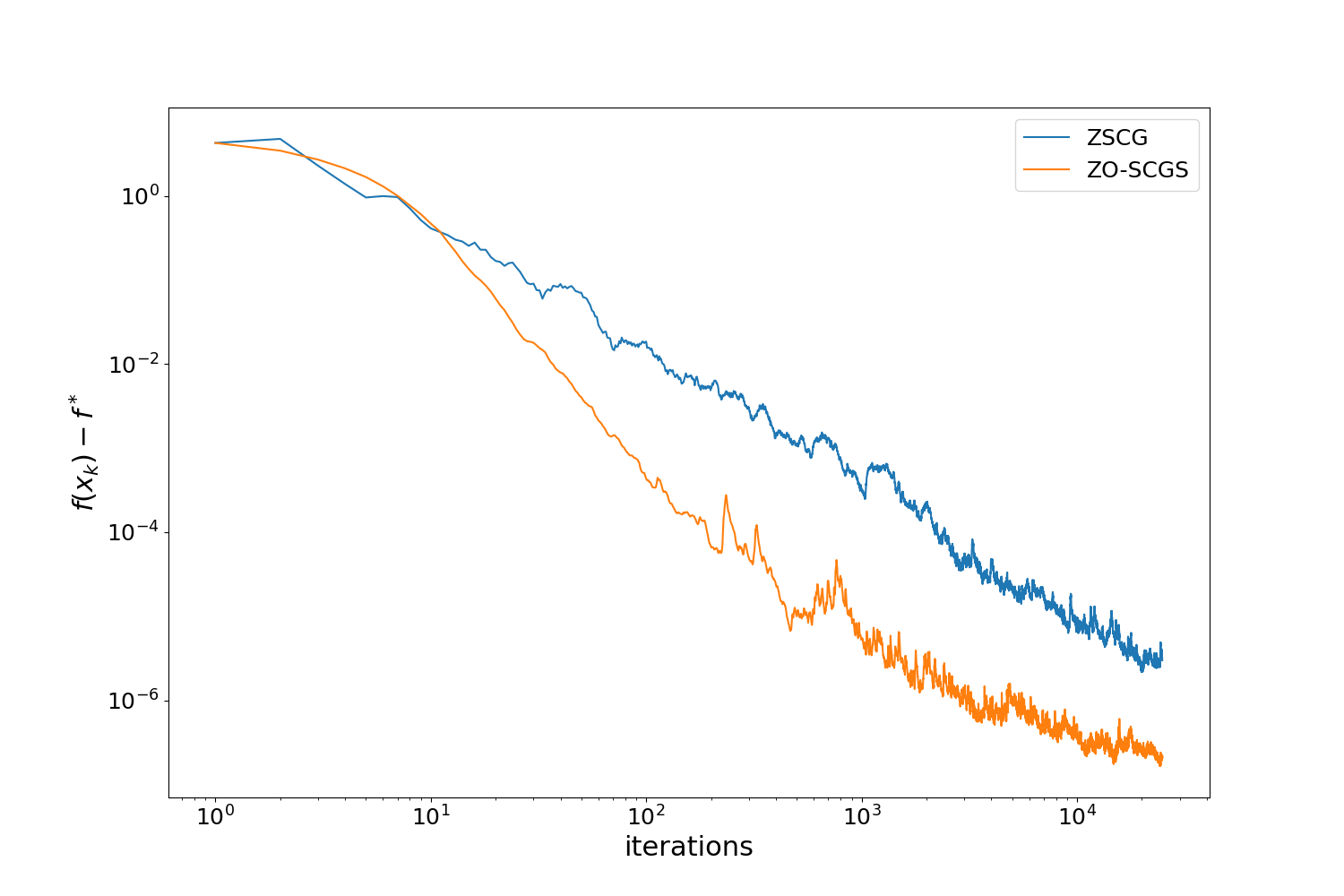}
    		\caption{}
    		\label{Figure_fixed_b}
    	\end{subfigure}
    	\caption{Comparison of convergence result of algorithms with fixed batch size. }
    	\label{Figure_fixed}
    \end{figure}
    
    \begin{figure}[t!]
    	\centering
    	\begin{subfigure}{0.49\textwidth}
    		\includegraphics[width=\textwidth]{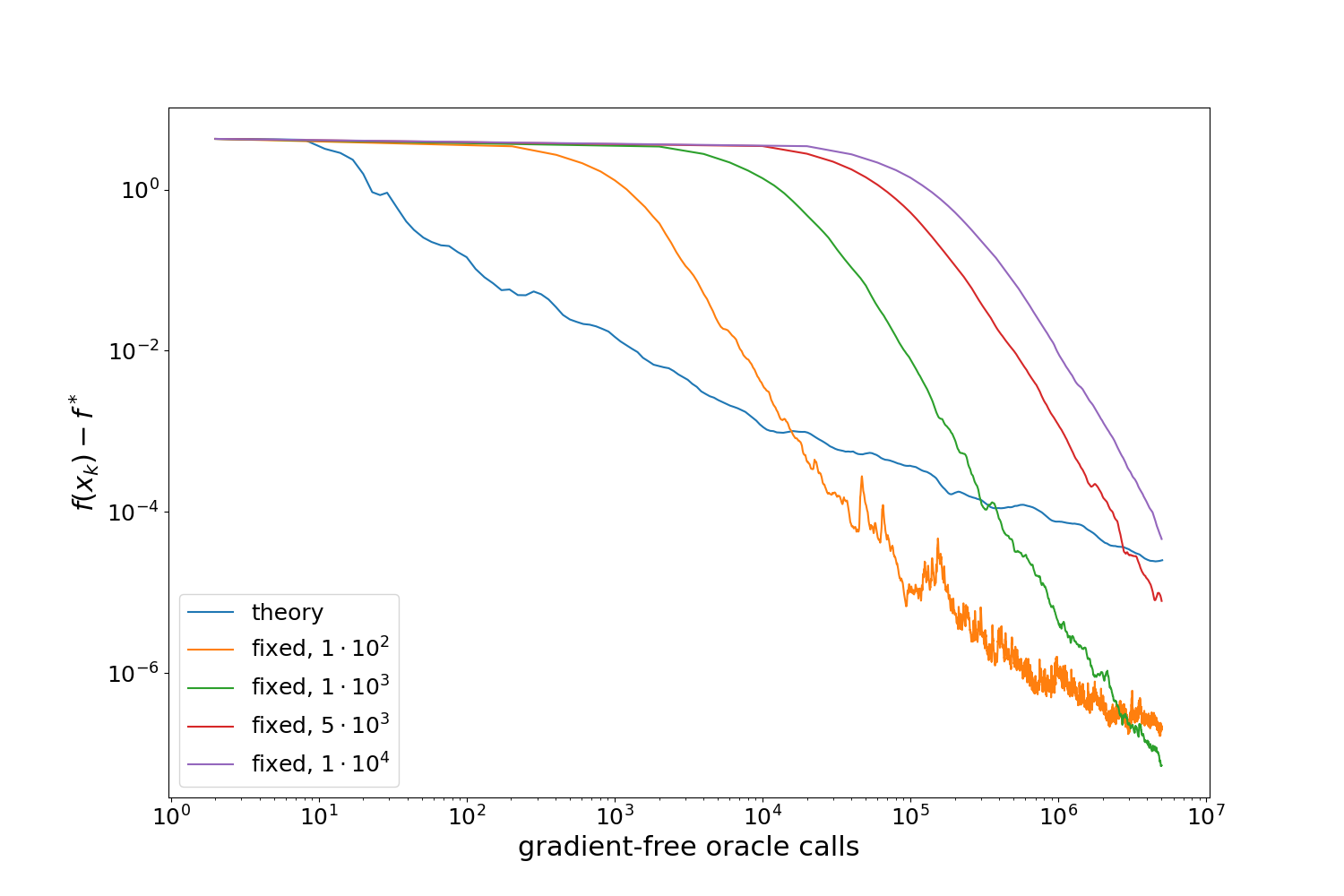}
    		\caption{}
    		\label{Figure_compare_a}
    	\end{subfigure}
    %%%%%%%%%%%%%%
    	\begin{subfigure}{0.49\textwidth}
    		\includegraphics[width=\textwidth]{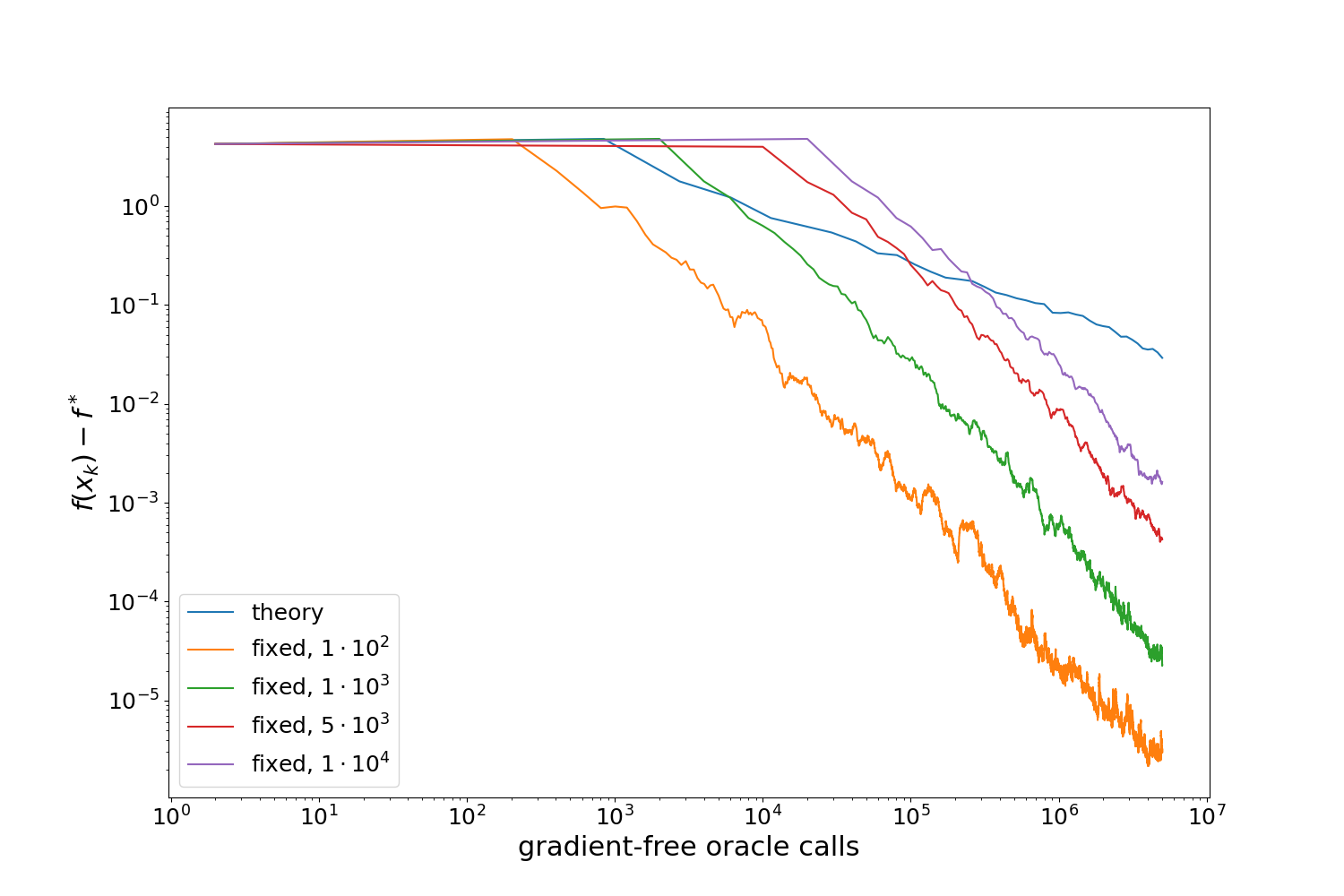}
    		\caption{}
    		\label{Figure_compare_b}
    	\end{subfigure}
    	\caption{Effect of the batch size parameter $B_k$ on convergence results.}
    	\label{Figure_compare}
    \end{figure}
    
    Figure \ref{Figure_compare} shows the effect of the batch size parameter $B_k$ on convergence. Where 'theory' means that the batch size corresponds to theoretical estimates, 'fixed $b$' means that the batch size corresponds to the value of $b$. Figure \ref{Figure_compare_a} explores the dependence of Zero-Order Stochastic Conditional Gradient Sliding (ZO-SCGS) on batch size $B_k$, and Figure \ref{Figure_compare_b} explores the dependence of the Zeroth-Order Stochastic Conditional Gradient Method (ZSCG) on batch size~$B_k$. We see that theoretical estimates of the batch size slow down the convergence rate of both methods. And we can also observe a tendency that the smaller the batch size, the faster the algorithms converge. However, it is worth observing the golden mean, because with a very small batch size the positive convergence effect will not be observed, as well as with a very large batch size.

\section{Conclusion}\label{section:Conclusion}
    We presented, as far as we know, the first gradient-free algorithm of the conditional gradient type, which is robust for solving non-smooth convex stochastic optimization problems (black-box problems in a non-smooth setting). Using a smoothing scheme with $l_2$ randomization and basing on an accelerated batched first-order algorithm, we showed that Zero-Order Stochastic Conditional Gradient Sliding (ZO-SCGS) is the optimal algorithm for three criteria: total number of iterations, oracle complexity, and maximum adversarial noise. Our theoretical results show that Algorithm \ref{alg:ZOSCGS} is a robust method not only for non-smooth black box problems, but also for black box problems with a smooth setting. We verified our theoretical results on a practical experiment in a smooth setup by comparing our algorithm with the state of the art algorithm. We have shown that using a fixed (small enough) batch size achieves better accuracy than with batch size derived from theoretical estimates.  


\newpage
\appendix

\section{Proof Theorem \ref{theorem_1}}\label{Appendix}
    Before giving the proof of Theorem \ref{theorem_1}, we prove the following statement from the last part of Subsection \ref{Subsection:smoothing_intuition}: to solve a non-smooth problem with $\varepsilon$-accuracy, it is necessary to solve a smooth problem with $(\varepsilon/2)$-accuracy, where $\varepsilon$-suboptimality is the accuracy of the solution in terms of expectation.\\
    \textbf{Connection between Problems \eqref{init_problem} and \eqref{problem_smoothed}}. Using Lemma \ref{Lemma:connect_f_with_f_gamma} we obtain
    \begin{eqnarray*}
        f(x^{N+1}) - f(x_{*}) &\overset{\circledOne}{\leq}& f_\gamma(x^{N+1}) - f(x_{*}) \overset{\circledTwo}{\leq} f_\gamma(x^{N+1}) - f_\gamma(x_{*}) + \gamma M_2  \\
        &\leq& f_{\gamma}(x^{N+1}) - f_\gamma(x_{*}(\gamma)) + \gamma M_2 \leq \frac{\varepsilon}{2 } + \frac{\varepsilon}{2} =\varepsilon,
    \end{eqnarray*}
    where $\circledOne$ means the first inequality of the Lemma \ref{Lemma:connect_f_with_f_gamma} and $\circledTwo$ means the second inequality of the Lemma \ref{Lemma:connect_f_with_f_gamma}. Thus, if we solve a smooth problem \eqref{problem_smoothed} with $\varepsilon/2$~accuracy ($f_{\gamma}(x^{N+1}) - f_\gamma(x_{*}(\gamma)) \leq \varepsilon/2$), we solve a non-smooth problem \eqref{init_problem} with $\varepsilon$-accuracy ($f(x^{N+1}) - f(x_{*}) \leq \varepsilon$) when we choose the gamma parameter~$\gamma = \varepsilon/ (2 M_2)$. 
    
    First, let's define the procedure CG (see Algorithm \ref{alg:CG}) to which Algorithm \ref{alg:ZOSCGS} refers. The Conditional Gradient procedure has the following input parameters: $g_0, u_0, \eta, \beta$ which correspond to parameters: $g_k, y_{k-1}, \eta_k, \beta_k$, where $k$ is number of iteration for Algorithm \ref{alg:ZOSCGS}. The CG procedure outputs the following value $u$.
    
    \begin{algorithm}
        \caption{Conditional Gradient procedure ($u \gets \text{CG}(g_0, u_0, \eta, \beta)$)}\label{alg:CG}
        \begin{algorithmic}[1]
        \For{$t=0,...,T$}
        \State $v_t \gets \text{argmin}_{v \in Q} \dotprod{g_t}{v}$
        \If {$\dotprod{g_t}{u_t - v_t} \leq \beta$}
        \State \Return $u_t$
        \EndIf
        \State $\alpha_t \gets \min \left\{ \frac{\dotprod{g_t}{u_t - v_t}}{\eta \| u_t - v_t \|^2}, 1 \right\}$
        \State $u_{t+1} \gets u_t + \alpha_t (v_t - u_t)$
        \State $g_{t+1} \gets g_0 + \eta (u_{t+1} - u_0)$
        \EndFor
        \end{algorithmic}
    \end{algorithm}
    We now obtain the convergence rate of Algorithm \ref{alg:ZOSCGS}: Zero-Order Stochastic Conditional Gradient Sliding (ZO-SCGS) to solve the non-smooth problem \eqref{init_problem}.\\
    \textbf{Convergence rate for Algorithm \eqref{alg:ZOSCGS}}. Let us write out the convergence rate for the Stochastic Conditional Gradient Sliding algorithm from \cite{Lan_2016}:
    \begin{equation}\label{eq:SCGS}
        \mathbb{E} \left[ f_\gamma(x_k) \right] - f_\gamma(x_*(\gamma)) \leq \frac{7.5 L D^2}{(k+1)(k+2)},
    \end{equation}
    where $L$ is a constant of Lipschitz gradient. Then, in order to obtain the convergence rate for Algorithm \ref{alg:ZOSCGS}, we use the Smoothing Scheme (see Subsection \ref{Subsection:smoothing_intuition} for more details). Substituting $L = L_{f_\gamma} = \frac{2 \sqrt{d} M M_2}{\varepsilon}$ (from Lemma \ref{Lemma:Lipschitz_gradient}) and $\sigma^2~\leq~2~\sqrt{2} \min \left\{ q, \ln d \right\} d^{2 - \frac{2}{p}}  M_2^2$ (from Remark \ref{Remark:sigma}) in \eqref{eq:SCGS} we obtain the convergence rate for Zero-Order Stochastic Conditional Gradient Sliding (ZO-SCGS):
    \begin{equation}\label{convergence_rate}
        \mathbb{E} \left[ f(x_k) \right] - f^* \leq \frac{15 \sqrt{d} M M_2 D^2}{\varepsilon (k+1)(k+2)}.
    \end{equation}
    We turn to estimates of  number of iterations and  call of the gradient-free oracle.\\
    \textbf{Iterative and oracular complexities.} We first find an estimate of the number of iterations $N$. To do this, we assume that the Algorithm \ref{alg:ZOSCGS} achieves $\varepsilon$-accuracy after $N$ iterations, then from \eqref{convergence_rate} we obtain:
    \begin{equation*}
        \frac{15 \sqrt{d} M M_2 D^2}{\varepsilon (N+1)(N+2)} \leq \varepsilon \;\;\; \Rightarrow \;\;\; N^2 \gtrsim \frac{15 \sqrt{d} M M_2 D^2}{\varepsilon^2} \;\;\; \Rightarrow
    \end{equation*}
    \begin{equation}\label{iteration_complexity}
        \Rightarrow \;\;\; N = \mathcal{O}\left( \frac{ d^{1/4} \sqrt{M M_2} D}{\varepsilon} \right).
    \end{equation}
    
    Next, we find the number of calls of the gradient-free oracle (see Definition~\ref{def:gradient_free_oracle}):
    \begin{eqnarray*}
        T = \sum_{k=1}^N B_k &=& \sum_{k=1}^N \frac{\min\{q,\ln~d\}d^{1-\frac{2}{p}} (k+3)^3 \varepsilon^2 }{(MD)^2} \\
        &\overset{\circledThree}{\simeq}& \frac{\min\{q,\ln~d\}d^{1-\frac{2}{p}} N^4 \varepsilon^2 }{4(MD)^2}\\ &\overset{\eqref{iteration_complexity}}{=}& \mathcal{O} \left( \frac{\min\{q,\ln~d\}d^{1-\frac{2}{p}} d M_2^2 M^2 D^4 \varepsilon^2 }{(MD)^2 \varepsilon^4} \right)\\
        &=& \mathcal{O} \left( \frac{\min\{q,\ln~d\}d^{2-\frac{2}{p}} M_2^2  D^2 }{ \varepsilon^2} \right)\\
        &=& 
            \begin{cases}
                 \mathcal{O} \left(  \frac{d M_2^2 D^2}{ \varepsilon^2}   \right), & p = 2 \;\;\; (q = 2);\\
                \mathcal{O} \left(  \frac{(\ln d) M_2^2 D^2}{ \varepsilon^2}   \right), & p = 1 \;\;\; (q = \infty),
            \end{cases}
    \end{eqnarray*}
    where in $\circledThree$ we use auxiliary fact that $\sum_{k=1}^N k^3 \simeq \frac{1}{4} N^4$. 
\end{document}